\title{On the Longest Common Pattern Contained in Two or More Random Permutations}
\author{Michael Earnest, University of Southern California \and Anant Godbole, East Tennessee State University \and  Yevgeniy Rudoy, The Johns Hopkins University}
\date{\today}

\documentclass[12pt]{article}
\usepackage{amsmath,amsfonts}
\usepackage{amsthm}
\usepackage{color}
\usepackage{verbatim}

\newcommand{\E}{\mathbb E}
\newcommand{\p}{\mathbb P}

\newcommand{\floor}[1]{\lfloor #1 \rfloor}

\newcommand{\e}{\varepsilon}
\newcommand{\Var}{\text{Var}}

\renewcommand{\(}{\left(}
\renewcommand{\)}{\right)}
\renewcommand{\th}{^\text{th}}

\newtheorem{theorem}{\\Theorem}
\newtheorem{lemma}{\\Lemma}

\begin{document}

\maketitle
\begin{abstract}
We provide upper and lower bounds for the expected length $\E(L_{n,m})$ of the longest common pattern contained in $m$ random permutations of length $n$.  We also address the tightness of the concentration of $L_{n,m}$ around $\E(L_{n,m})$.
\end{abstract}

\section{Introduction}
Let $\pi=\pi_1\pi_2\ldots\pi_n$ and $\sigma=\sigma_1\sigma_2\ldots\sigma_k$ be permutations, where $k<n$. We say that $\sigma$ is a pattern \emph{contained in} $\pi$ if there are a series of entries in $\pi$, not necessarily consecutive, which have the same relative ordering as $\sigma$. For instance, when $\pi=153642$ and $\sigma=132$, we see that $\pi$ contains the subsequence 1,6,4, which has the same relative ordering as 1,3,2 in the sense that both have their smallest entry first, their largest second, and their middle last. Thus $\pi=153642$ contains the pattern $\sigma=132$.  Formally, we say $\sigma$ is a pattern contained in $\pi$ if there exist indices $i_1<i_2<\ldots < i_k$  such that the sequence $\pi_{i_1}\pi_{i_2}\ldots\pi_{i_k}$ is \emph{order isomorphic} to $\sigma$; i.e., for any $a$ and $b$, $\pi_{i_a}<\pi_{i_b}$ if and only if $\sigma_a<\sigma_b$.

Given $m$ permutations $\pi_1,\dots,\pi_m\in S_n$, a \emph{common pattern} is a permutation $\sigma$ which is a pattern contained in all of the $\pi_i$, and a \emph{longest common pattern} is a common pattern of maximum length.  Define $L_{n,m}$ to be the length of the longest common pattern (LCP) contained in $m$ uniformly randomly chosen permutations of length $n$. Our main results demonstrate that $\E(L_{n,m})\in\Theta(n^\frac{m}{2m-1})$ as $n\to\infty$, and we give asymptotic bounds for $\E(L_{n,m})$.  We also show that $L_{n,m}$ is concentrated in an interval of length $O(n^{\frac{m}{4m-2}})$ around $\E(L_{n,m})$.

The topics contained in this paper have obvious connections and similarities to two classical and well-studied problems, namely those of the longest common subsequence (LCS) $N_n$ between two random strings \cite{steele} and the longest monotone subsequence (LMS) of a random permutation \cite{ad}.   Here is a summary of key results in these areas:   

First, consider the LCS problem.  Given two independent, identically distributed binary strings $(X_1,\ldots,X_n)$ and $(Y_1,\ldots,Y_n)$, subadditivity arguments yield that 
\[\frac{\E(N_n)}{n}\to c\]
for some constant $c\in(0,1)$. The value of $c$ is not known to date (see \cite{steele}), and the best currently known bounds appear to be from \cite{lueker}, namely $0.7880\le c\le0.8296$. The situation where the variables take values from an alphabet $\{0,1,\ldots,d-1\}$  of size $d$ is similarly in an incompletely understood state, though techniques such as Azuma's and Talagrand's inequalities \cite{as} have been used to provide estimates of the width of concentration intervals of the LCS around its mean for all alphabet sizes.  The work of Kiwi, Loebl, and Matou\u sek \cite{klm} is of particular relevance to this paper.  They consider the case of large alphabet sizes and verify that the limiting constant $c_d$ in the alphabet $d$ LCS problem does indeed satisfy 
\[\lim_{d\to\infty}c_d\sqrt{d}=2,\]
as conjectured by Sankoff and Manville.

Moving to the LMS problem, we can do no better than to cite \cite{steele}, \cite{bdj}, and \cite{ad}, which takes us from the early years of the Erd\H os-Szekeres theorem (every permutation of $[n^2+1]$ contains a monotone sequence of length $n+1$), to the work of Logan-Shepp-Vershik-Kerov (namely that the longest monotone subsequence of a random permutation on $[n]$ is asymptotic to $2\sqrt{n}$), to concentration results (Janson, Kim, Frieze) that reveal that the standard deviation of the size of the LMS is of order $\Theta(n^{\frac{1}{6}})$, and culminating with the work of \cite{bdj} that exhibit the limiting law of a normalized version of the LMS.

Other forms of LMS problems have been considered in \cite{stanley} and \cite{albert}, and algorithmic results on the LCP problem that we study in the subsequent sections may be found in \cite{br} and \cite {brv}.

\section{Upper Bound}
\begin{theorem}
$\E(L_{n,m})\le\lceil en^\frac{m}{2m-1}\rceil.$
\end{theorem}
\begin{proof}
First, we provide an upper bound on $\p(L_{n,m}\ge k)$, when $k>en^\frac{m}{2m-1}$. Let $S_1,\dots,S_m$ be subsets of $[n]$, each of size $k$. These define $m$ subsequences of $\pi_1,\dots,\pi_m$, where $S_i$ corresponds to the indices of the subsequence $\pi_i$. Since the $\pi_i$ are independent, the orderings of the subsequences will also be independent, and as each subsequence has $k!$ possible equally likely orderings, the probability that the subsequences will be order isomorphic is $1/(k!)^{m-1}$. Furthermore, $L_{n,m}\ge k$ if and only if these subsequences are order isomorphic for at least one of the  $\binom{n}{k}^m$ choices for the list $S_1,\dots,S_m$, so
\[
\p(L_{n,m}\ge k) \le \binom{n}k^m\frac1{k!^{m-1}} = \left(\frac{n^k}{k!}\right)^m\frac1{k!^{m-1}} = 
\frac{n^{mk}}{k!^{2m-1}} ,
\]
Using the bound $k!>\sqrt{2\pi k}\(\frac{k}e\)^k$, this implies
\[
\p(L_{n,m}\ge k)\le \frac1{(2\pi k)^{(m-\frac12)}}\left(\frac{e^{2m-1}n^m}{k^{2m-1}}\right)^{k},
\]
and thus when $k>en^\frac{m}{2m-1}$, we have $\p(L_{n,m}\ge k)\le (2\pi k)^{-(m-\frac12)}. $

We can write $\E(L_{n,m})$ as
\begin{align*}
\E(L_{n,m})=\sum_{k=1}^{n} \p(L_{n,m}\ge k)
&=\sum_{k=1}^{\left\lfloor en^\frac{m}{2m-1}\right \rfloor} \p(L_{n,m}\ge k)+\sum_{k=\lceil en^\frac{m}{2m-1}\rceil}^n \p(L_{n,m}\ge k)\\
&\le \lfloor en^\frac{m}{2m-1}\rfloor + \sum_{k=\lceil en^\frac{m}{2m-1}\rceil}^n \frac1{(2\pi k)^{(m-\frac12)}}\\
&\le \lfloor en^\frac{m}{2m-1}\rfloor + \frac1{(2\pi)^{(m-\frac12)}}\sum_{k=1}^\infty \frac1{k^{(m-\frac12)}}\\
\end{align*}
The second term on  the last line is known to have a sum less than 1 for $m \ge 2$, so it follows that $\E(L_{n,m})\le\lceil en^\frac{m}{2m-1}\rceil$.	
\end{proof}

\section{Lower Bound}

The purpose of this section is to prove the following asymptotic lower bound for $\E(L_{n,m})$.
\begin{theorem}
$\liminf_{n\to\infty} \frac{\E(L_{n,m})}{n^{\frac{m}{2m-1}}}\ge \frac12.$
\end{theorem}
The proof will require developing some machinery. We first give a method to generate the $m$ random permutations, which will allow us to identify common patterns more easily. 

\begin{proof}
Let $I$ be the interval $[0,1]$. By choosing $n$ points uniformly randomly in the unit square $I^2$, we can specify a permutation $\pi\in S_n$ uniformly at random as follows. Consider the point with the $i\th$ smallest $x$ coordinate. Assign $\pi(i)=j$ if that point has the $j\th$ smallest $y$ coordinate. For our proof, let $X_{i,j}$, for $1\le i\le m$ and $1\le j\le n$, each be chosen uniformly from $I^2$, and let $\pi_i$ for $1\le i\le m$, be the permutation specified in the above fashion by the points $X_{i,1},\dots,X_{i,n}$.

Furthermore, let $r=\lfloor n^\frac1{2m-1}\rfloor$. We can parition $I^2$ into a $r^m$ by $r^m$ array of $r^{2m}$ equally sized square boxes. Call a box \emph{full} if, for each $1\le i\le m$, it contains at least one point from the set $\{X_{i,j}:1\le j\le n\}$. In other words, it contains a point used to to define each of the $m$ permutations. Furthermore, define a \emph{scattering} to be a set of full boxes, where each pair are in a different row and column.  Scatterings are related to common patterns among $\pi_1,\dots,\pi_m$ as follows: if there is a scattering of size $k$, there will be a common pattern among $\pi_1,\dots,\pi_m$ of length $k$. This can be seen by examining the $m$ subsequences defined by the points in these full boxes. In this proof, we find a probabilistic lower bound for the number of full boxes, and use this to find a lower bound for the expected size of the largest scattering.

We need one last tool. Let $\rho$ be a random ordering of the $r^{2m}$ boxes, so that $\rho$ is distributed uniformly over the $(r^{2m})!$ bijections from $\{1,2,\dots,r^{2m}\}$to  the set of $r^{2m}$ boxes. Ordering the boxes randomly (as opposed to some arbitrary, deterministic ordering) will simplify parts of the proof later, which will consider the boxes in the order defined by $\rho$. Finally, let $F_i$ be the event that $\rho(i)$ is full, and let $F$ be the total number of full boxes.

\begin{lemma}
For all $\e>0$, $\p(F<(1-\e)r^m)\to0$ as $n,r\to\infty$.
\end{lemma}
\begin{proof}

Since $n\ge r^{2m-1}$, we have

\begin{align*}
\p(F_{i}) &= \(1-\(\frac{r^{2m}-1}{r^{2m}}\)^{n}\)^m\\
&\ge \(1-\(1-\frac{1}{r^{2m}}\)^{r^{2m-1}}\)^m\\
&\ge \(1-e^{-1/r}\)^m\\
&\ge \(\frac1r-\frac1{2r^2}\)^m\\
&\ge\frac1{r^m}\(1-\frac{m}{2r}\).
\end{align*}

Then, by linearity of expectation, we have $\E(F)\ge\(1-\frac{m}{2r}\)r^m$. Also, the inequalities $e^{-x}\ge1-x\ge e^{-x/(1-x)}$ show that
\begin{align*}
\p(F_{i}) &= \(1-\(\frac{r^{2m}-1}{r^{2m}}\)^{n}\)^m\\
&\le \(1-e^{-n/(r^{2m}-1)}\)^m\\
&\le \(\frac{n}{r^{2m}-1}\)^m\\
&\le\frac1{r^m}(1+o(1)).
\end{align*}
We now give a bound for $\Var(F)$. Notice that the indicator variables $1_{F_i}$ for $F_i$ are pairwise negatively correlated; given that $F_i$ has occurred, it is less likely that $F_j$ will occur (since there will be strictly fewer points that can land in the $j\th$ box).  Thus
\[
\Var(F)<
\sum_{i=1}^{r^{2m}}\Var(1_{F_i})
<\sum_{i=1}^{r^{2m}}\p(F_i)
< r^m(1+o(1)).
\]
Then, for any $\e>0$, we have that
\begin{align*}
\p(F<(1-\e)r^m)&=\p\(F<\(1-\frac{m}{2r}\)r^m-\(\e-\frac{m}{2r}\)r^m\)\\
&\le \p\(F<\E(F)-\(\e-\frac{m}{2r}\)r^m\)\\
&\le \p\(|F-\E(F)|>\(\e-\frac{m}{2r}\)r^m\).
\end{align*}
If we choose $n$ sufficiently large so $\frac{m}{2r}<\e$, then by Chebychev's inequality, we have
\[
\p(F<(1-\e)r^m)\le \frac{\Var(F)}{\(\e-\frac{m}{2r}\)^2r^{2m}}\le \frac{1+o(1)}{\(\e-\frac{m}{2r}\)^2r^{m}}\to0
\]
as $r\to\infty$.
\end{proof}

Given that there are $F$ full boxes, index them with the numbers $1$ through $F$ in the same order as $\rho$, and let $B_k$ refer to the $k\th$ full box. The fact that $\rho$ was a random ordering ensures that, given $\{B_1,\dots,B_{k-1}\}$, $B_k$ is distributed uniformly among the $r^{2m}-k+1$ locations not occupied by $\{B_1,\dots,B_{k-1}\}$. Define the sequence of random variables $\{S_k\}_0^{F}$, where $S_0=0$ and $S_k$ is the size of the largest scattering which is a subset of $\{B_1,\dots, B_{k}\}$. Then $S_1=1$, and $S_{k+1}$ is equal to either $S_k$ or $S_k+1$.

Let $\e>0$ be given. Throughout the rest of this proof, we will use the expression $S_x$ to mean $S_{\lfloor x\rfloor}$. The next lemma formalizes the previous observation that given a size $k$ scattering, there will be a common pattern of length $k$.

\begin{lemma}
For large enough $n$, $\E(L_{n,m})\ge \E(S_{(1-\e)r^m})(1-o(1)).$ 
\end{lemma}

\begin{proof}
For ease of reading, let $r^m=R$. By conditioning $L_{n,m}$ on the event $F>(1-\e)R$, we have
\[
\E(L_{n,m})\ge \E(L_{n,m}|F>(1-\e)R)\cdot \p(F>(1-\e)R).
\]
 Given $F>(1-\e)R$, the variable $S_{(1-\e)R}$ is well defined. Suppose that $S_{(1-\e)R}=k$, so that there exists a scattering of size $k$. The centers of these $k$ boxes define a permutation $\sigma\in S_k$, as described in the beginning of this section. For any $i\in 1,\dots,m$, since the boxes in the scattering are full, there will be a subsequence $\pi_i(j_1),\dots,\pi_i(j_k)$, where the points corresponding to each entry will be in different boxes in the scattering. This implies the subsequence is order isomorphic to $\sigma$. Thus, $\sigma$ is a common pattern among $\pi_1,\dots,\pi_m$ of length $k$, implying $L_{n,m}\ge S_{(1-\e)R}$.  Combining this with the proof of Lemma 1, which guarantees $P(F>(1-\e)R)\ge1-\frac{C}{r^m}$ for some constant $C$ and large $n$, we get that 
\[\E(L_{n,m})\ge \E(L_{n,m}|F>(1-\e)R)\cdot \p(F>(1-\e)R)\ge\E(S_{(1-\e)R})\(1-\frac{C}{r^m}\),\] as asserted.
\end{proof}

For the rest of the proof, we will assume $F>(1-\e)R$, so that $S_{(1-\e)R}$ is well defined. The next lemma provides a lower bound for $\E(S_k)$ in terms of another sequence.

\begin{lemma}
For $R=r^m$, define the sequence $\{y_k\}_{k=0}^R$, where $y_0=0$, and
\begin{equation}\label{yrec}
y_{k+1}=y_k+\frac1R(1-y_k)^2
\end{equation}
Then, for all $0\le k\le (1-\e)R$,
$$\frac{\E(S_k)}{R}\ge y_k.$$ 
\end{lemma}
\begin{proof}
Given $S_k=s$, there is scattering, $T$, where $|T|=s$. Notice that $B_{k+1}$ can be appended to $T$ to make a larger scattering if it is in one of the $(R-s)^2$ locations not sharing a row or column with any box in $T$, in which case there will exist a scattering of length $k+1$. This occurs with probability $\frac{(R-s)^2}{R^2-k}$, so that
\begin{eqnarray*}
\E(S_{k+1})&=&\E(S_k)+\E(S_{k+1}-S_k)\\
&=&\E(S_k)+\p(S_{k+1}-S_k=1)\\ 
&=&\E(S_k)+\sum_s\p(S_k=s)\frac{(R-s)^2}{R^2-k}\\
&\ge&\E(S_k)+\sum_s\p(S_k=s)\(1-\frac{s}{R}\)^2\\
&=&\E(S_k)+\E\(\(1-\frac{S_k}{R}\)^2\)\\
&\ge&\E(S_k)+\(\E\(1-\frac{S_k}{R}\)\)^2,
\end{eqnarray*}
and thus
\begin{equation}\label{Erec}
\frac{\E(S_{k+1})}R \ge \frac{\E(S_k)}R+\frac1R\(1-\frac{\E(S_k)}R\)^2.
\end{equation}
We now use induction to complete the proof.  Evidently $S_0=y_0=0$ and ${\E(S_1)}/{R}={1}/{R}= y_1$.  Assume that $\E(S_k)/R\ge y_k$.  Then, we note that the right hand side $f(x)$ of (2) is an increasing the function of the argument $x:=\E(S_k)/R$ since $f'(x)=1-\frac2R(1-x)>0$ if $R\ge2$.  It follows from the induction hypothesis that $$\frac{\E(S_{k+1})}{R}\ge y_k+\frac1R(1-y_k)^2=y_{k+1}.$$
\end{proof}

\begin{lemma}
$\lim_{R\to\infty} y_{\lfloor(1-\e)R\rfloor}=\frac{1-\e}{2-\e}.$
\end{lemma}

\begin{proof} The sequence $y_k$ is (coincidentally) the result of applying Euler's method to approximate the solution to the differential equation $y'(x)=(1-y)^2$, with initial condition $y(0)=0$, using step size $1/R$. This has a unique solution on the interval $(0,1)$, given by $y(x)=\frac{x}{x+1}$. 

To prove this Lemma, we cite Theorems 1.1 and 1.2 of \cite{greenspan}, which proves that the error terms for Euler's method converge uniformly to zero. The only difficulty is that this proof assumes that the DE is of the form $y'=F(x,y)$, with $\frac{\partial F}{\partial y}$ being bounded for all $y\in\mathbb{R}$. In our case, $\frac{\partial}{\partial y}(1-y)^2$ is not bounded. However, a careful examination of the proof shows that, if $y_k,y(x)\in[a, b]$ for all $k$ and $x\in[0,1]$, it is only required that $|\frac{\partial F}{\partial y}|<M$ for $y\in [a,b]$. Clearly $y(x)=\frac{x}{x+1}\in[0,1]$ for $x\in[0,1]$, and it can be shown by induction that $y_k\in[0,1]$ for $0\le k\le R$. Thus, since $\frac{\partial}{\partial y}(1-y)^2$ is bounded on $[0,1]$, the proof still applies.

In this case, the $k\th$ error term is $|y_k-y(k/r)|$, so that 
$$
\lim_{R\to\infty} \,\,\,y_{\lfloor(1-\e)R\rfloor}-y\(\frac{\lfloor(1-\e)R\rfloor}{R}\)=0.
$$
Since $y(1-\e-\frac1R)\le y\(\frac{\lfloor(1-\e)R\rfloor}{R}\)\le y(1-\e)$, this proves that $\lim_{R\to\infty}y_{\lfloor(1-\e)R\rfloor}= y(1-\e)=\frac{1-\e}{2-\e}$.\hfill\end{proof}

Finally, combining Lemmas 2, 3 and 4, we get
\[
\liminf_{n\to\infty} \frac{E(L_{n,m})}{R}\ge\liminf_{R\to \infty} \frac{E(S_{(1-\e)R})}R \ge \lim_{R\to \infty}y_{\floor{(1-\e)R}}=\frac{1-\e}{2-\e}.
\]
Since this holds for all $\e>0$, this implies $\liminf_{n\to\infty} \frac{E(L_{n,m})}{R}\ge\frac12$. Since $
\lim_{n\to\infty}{n^\frac{m}{2m-1}}/{R}=1$ (recall that $R=\floor{n^\frac{1}{2m-1}}^m$), we finally have that $\liminf_{n\to\infty}{E(L_{n,m})}/{n^\frac{m}{2m-1}}\ge\frac12$.
\end{proof}

This lower bound can actually be improved by adjusting the preceding proof slightly. At the beginning of the proof, we divided $I^2$ into a $r^m$ by $r^m$ grid of smaller squares; if we instead use a $c_mr^m$ by $c_mr^m$ grid, for some constant $c_m$, then we obtain the lower bound 
$$
\liminf_{n\to\infty}\frac{E(L_{n,m} )}{n^{m/2m-1}}\ge\frac{c_m}{1+(c_m)^{2m-1}}.
$$
In particular, letting $c_m=\(\frac1{2m-2}\)^\frac1{2m-1}$ shows 
$$
\liminf_{n\to\infty}\frac{E(L_{n,m} )}{n^{m/2m-1}}\ge\frac{2m-2}{2m-1}\(\frac1{2m-2}\)^\frac1{2m-1},
$$
which equals 0.529 for $m=2$, and converges to 1 as $m\to\infty$.  In addition, we have conducted analyses that reveal the following promising methods for improvements:  (i) Poisson approximation \cite{bhj}; (ii) coding the problem using large alphabet results \cite{klm}; (iii) exploiting the possibility of multiple matchings within cells; and (iv) exploiting the theory of perfect matchings in random bipartite graphs \cite{frieze}.

\section{Concentration Around the Mean}  In this section, we use Talagrand's inequality as in \cite{as} to show that $L_{n,m}$ is concentrated in an interval of length $O(n^{m/(4m-2)})$ around $\E(L_{n,m})$:  For each $i\in[m]$, define a sequence $\{X_{i,j}\}_{j=1}^n$ of independent and identically distributed random variables uniformly distributed on $[0,1]$; the order statistics of each sequence will model the a random permutation in $S_n$.  It is evident that the quantity $L_{n,m}$ is 1-Lipschitz, in the sense that altering one of the $mn$ random variables can change $L_{n,m}$ by at most one.  Also, the event $\{L_{n,m}\ge b\}$ can be ``certified" by the values of $bm$ of the random variables.  It follows by Theorem 7.7.1 in \cite{as} that for each $b,t$,
\begin{equation}\p(L_{n,m}\le b-t{\sqrt{mb}})\p(L_{n,m}\ge b)\le\exp\{-t^2/4\}.\end{equation}
Setting  $b={\rm Med}(L_{n,m})$ in (3) yields, for any $t\to\infty$,
\[\p(L_{n,m}\le {\rm Med}(L_{n,m}n)-t{\sqrt{{m\cdot\rm Med}(L_{n,m})}})\to0,\]
and the same inequality, with $b-t{\sqrt{mb}}={\rm Med}(L_{n,m})$ gives
\[\p(L_{n,m}\ge {\rm Med}(L_{n,m})+t{\sqrt{m^2t^2+m\cdot{\rm Med}(L_{n,m})}})\to0,\]
which together imply a concentration in an interval of width ${\sqrt{{m\cdot \rm Med}(L_{n,m})}}$ around ${\rm Med}(L_{n,m})$.  The proof is completed by invoking an inequality such as the one in Fact 10.1 in \cite{mr}, which implies that
\[\vert\E(L_{n,m})-{\rm Med}(L_{n,m})\vert\le40{\sqrt{m\E(L_{n,m})}},\]and noting that $\E(L_{n,m})=\Theta(n^{\frac{m}{2m-1}}).$  \hfill\qed

\section{Open Problems} Several problems come immediately to mind, and most concern finding analogs of classical results on the LCS and LMS problems.  First and foremost, can subadditivity or monotonicity somehow be invoked to show that $$\lim_{n\to\infty} \frac{\E(L_{n,m})}{n^{\frac{m}{2m-1}}}$$ exists, and if so, what is the limiting constant?  Second, what is the ``correct" interval of concentration of $L_{n,m}$ around its mean?  Thirdly, what can be said, \`a la Baik, Deift and Johansson \cite{bdj}, about the appropriately normalized limiting distribution of $L_{n,m}$?  Lastly, is our conjecture (inspired by work in \cite{klm}) that
\[\lim_{m\to\infty}\lim_{n\to\infty}\frac{\E(L_{n,m})}{n^{\frac{m}{2m-1}}}=2\]true?
\section{Acknowledgments}  The research of all three authors was supported by NSF Grant 1004624 and conducted during the Summer 2012 REU program at East Tennessee State University.


\begin{thebibliography}{99}



\bibitem{albert} M. Albert (2007), ``On the length of the longest subsequence avoiding an arbitrary pattern in a random permutation," {\it Rand. Structures Alg.} {\bf 31}, 227--238.
\bibitem{ad} D. Aldous and P. Diaconis  (1999), ``Longest increasing subsequences: from patience sorting to the Baik-Deift-Johansson theorem,"
{\it Bull. Amer. Math. Soc.} {\bf 36}, 413--432.
\bibitem{as} N. Alon and J. Spencer (2000), {\it The Probabilistic Method, 2nd Edition}, Wiley, New York.
\bibitem{bdj} J. Baik, P. Deift, and K. Johansson (1999), ``On the distribution of the length of the longest increasing subsequence of random permutations," {\it  J. Amer. Math. Soc.} {\bf 12}, 1119--1178.
\bibitem{bhj} A.Barbour, L. Holst and S. Janson (1992), {\it Poisson Approximation,} Oxford University Press.
\bibitem{br} M. Bouvel and D. Rossin (2006),  ``The longest common pattern problem for two
permutations," {\it Pure Mathematics and Application} {\bf 17}, 55--69.
\bibitem{brv} M. Bouvel, D. Rossin, and S. Vialette (2007), ``Longest common separable pattern among permutations," in {\it Combinatorial Pattern Matching,
Lecture Notes in Computer Science} {\bf 4580}, 316--327.
\bibitem{frieze} A. Frieze (2005), ``Perfect matchings in random bipartite graphs with minimal degree at least 2," {\it Rand. Structures Alg.} {\bf 26}, 319--358.
\bibitem{greenspan} D. Greenspan (2006), {\it Numerical Solutions of Ordinary Differential Equations.} Wiley, Weinheim.
\bibitem{klm}  M. Kiwi, M. Loebl, and J. Matou\u sek (2004), ``Expected length of the longest common subsequence for large alphabets," {\it Lecture Notes in Computer Science} {\bf 2976}, 302--311
\bibitem{lueker} G. Lueker (2009), ``Improved bounds on the average length of longest common subsequences," {\it J. Assoc. Computing Machinery} {\bf 56}, Article 17, 38 pages.
\bibitem{mr} M. Molloy and B. Reed (2002), {\it Graph Colouring and the Probabilistic Method,} Springer Verlag, Berlin.
\bibitem{stanley} R. Stanley (2008), ``Longest alternating subsequences of permutations," {\it Michigan Math. J.} {\bf 57} 675--687. 
\bibitem{steele} J. M. Steele (1987),  {\it Probability Theory and Combinatorial Optimization}, SIAM, Philadelphia.

\end{thebibliography}
\end{document}